\documentclass[11pt]{article}
\usepackage{amsmath,amsfonts,amssymb,amsthm,amscd}

\newtheorem{thm}{Theorem}[section]
\newtheorem{lem}[thm]{Lemma}

\newtheorem{cor}[thm]{Corollary}

\newtheorem{claim}[thm]{Claim}

\newtheorem{rk}[thm]{Remark}
\newtheorem{definition}[thm]{Definition}
\newtheorem{example}[thm]{Example}
\newtheorem{lemma}[thm]{Lemma}

\newcommand{\T}{\mbox{${\mathbb  T}$}}
\newcommand{\CH}{\mbox{$\mathcal H$}}
\newcommand{\CM}{\mbox{$\mathcal M$}}
\newcommand{\CN}{\mbox{$\mathcal N$}}
\newcommand{\R}{\mbox{${\mathbb  R}$}}                  
\newcommand{\Na}{\mbox{${\mathbb  N}$}}                 
\newcommand{\Z}{\mbox{${\mathbb  Z}$}}                  
\newcommand{\Q}{\mbox{${\mathbb  Q}$}}              
\newcommand{\G}{\mbox{${\mathbb  G}$}}
\newcommand{\gp}{\mbox{${\mathbb  F}_{p}$}}             

\newcommand{\sub}{\subseteq}

\newcommand{{\ff}}{\mbox{$\varphi$}}                   
\newcommand{{\la}}{\mbox{$\lambda$}}                   
\newcommand{\lan}{\langle}
\newcommand{\ran}{\rangle}

\newcommand{\M}{\mbox{$\cal M$}}

\title{$G$-linear sets and torsion points in definably compact groups}
\author{Margarita
Otero\thanks{Partially supported by GEOR MTM2005-02568 and Grupos UCM 910444}\\Universidad Aut\'onoma de Madrid\and Ya'acov
Peterzil\\University of Haifa}

\begin{document}

\maketitle

\begin{abstract} Let $G$ be a definably compact group 
in an o-minimal expansion of a real closed field. We prove that if
$\dim(G\setminus X)< \dim G$ for some definable $X\sub G$ then $X$
contains a torsion point of $G$.  Along the way we develop a general
theory for so-called {\em $G$-linear sets}, and investigate
definable sets which contain abstract subgroups of
$G$.\end{abstract}



\section{Introduction}

We prove:

\begin{thm} Let $G$ be a definably compact group in an
o-minimal expansion of a real closed field. If  $X\sub G$ is a
definable large set (i.e. $\dim(G\setminus X)<\dim G$) then $X$
contains a torsion point of $G$.
\end{thm}

 This is a weak approximation to the
conjecture that every generic set in a definably compact group
contains a torsion point, a conjecture which itself follows from the
``compact domination conjecture'' from \cite{07HPP}.

The proof starts with the abelian case, where we first develop, in
 o-minimal expansions of an ordered group, a notion of a coset of a definable local
subgroup (a $G$-linear set). We then use tools from Algebric
Topology, in expansions of real closed fields, to conclude the
theorem.

\section{$G$-linear sets}
In this section \M\  will denote an
 o-minimal expansion of an ordered group. By   definable we mean definable in \M. $G$ will denote a definable   abelian group. All topological concepts  are with respect to the group topology such $G$ is equipped with. (We
believe that most, if not all, of the work below can be  developed
in an arbitrary definable group, but for simplicity we limit
ourselves to the commutative setting.)

\begin{definition} \emph{1)} Given two sets $X,Y\subseteq G$ and $a\in G$, we say that
{\em  $X$ and $Y$ have the same germ at $a$}, in notation $X=_a Y$,
if there exists an open neighborhood $U$ of $a$ such that $X\cap
U=Y\cap U$. We say that {\em the  germ of $X$ at $a$ is contained in the
germ of $Y$ at $a$}, in notation $X\sub_a Y$, if there exists an open $U\ni a$
such that $X\cap U\subseteq Y\cap U$.

\emph{2)} Given $X\subseteq G$ and $g,h\in G$, we say that {\em
the germ of $X$ at $g$ is $G$-equivalent to the germ of $X$ at
$h$} if $X-g=_0 X-h$. (Note that if $X$ is a definable set then we
obtain in this way a definable equivalence relation on $G$.)
\end{definition}

\begin{definition} Let $X$ be a definable subset of $G$.

\emph{1)} $X$  is called {\em $G$-linear} if for every $g,h\in X$ we
have $X-g=_0 X-h$. Given $g\in X$, we say that {\em $X$ is locally
$G$-linear at $g$} if there exists an open $U\ni g$ such that for
every $h\in U\cap X$ we have $X-g=_0 X-h$.

Two $G$-linear sets $X,Y\sub G$ are called {\em $G$-equivalent} if
for every $g\in X$, $h\in Y$, we have $X-g=_0 Y-h$.

 \emph{2)} Let $X$ be a $G$-linear set. A definable $Y\sub G$
is called {\em a $G$-subset of $X$}  if for every $h\in Y$ and
$g\in X$, we have $Y-h\sub_0 X-g$.
\end{definition}

\begin{example}
Let $G$ be a definable abelian group and $H$  a definable
subgroup. If $X$ is a relatively open subset of a coset of $H$
then $X$ is $G$-linear. If $Y$ is any definable subset of a coset
of $H$ then $Y$ is a $G$-linear subset of $X$.
\end{example}

We observe without a proof (we will not be using it): If $\CM$
expands a real closed field then $Y$ is a $G$-subset of the
$G$-linear set $X$ if and only if at every smooth point $y\in Y$,
the tangent space of $Y$ at $y$ is a subset of the tangent space of
$X$ at any of its points point.

\begin{lemma}\label{linear0}
If $X$ is definably connected and locally $G$-linear at every $g\in
X$ then $X$ is $G$-linear.
\end{lemma}
\begin{proof} Because $X$ is locally $G$-linear at every point the $G$-equivalence class of
every $g\in X$ is relatively open in $X$. But then it is also
relatively closed (since its complement is a union of open sets), so
by definable connectedness there exists only one $G$-equivalence
class, so $X$ is $G$-linear.\end{proof}

\begin{definition} Given $X\sub G$, we let
$$X_{lin}=\{g\in X: \mbox{ $X$ is locally $G$-linear at }g\}.$$
\end{definition}

Notice that $X_{lin}$ is a relatively open (possibly empty) subset
of $X$. By the last lemma, every definably connected component of
$X_{lin}$ is a $G$-linear set, but $X_{lin}$ itself might not be
$G$-linear (E.g., if  $X=H_1\cup H_2$ is the union of two definable
subgroups of $G$, none containing the other,  then
$X_{lin}=X\setminus (H_1\cap H_2)$ is not a $G$-linear set). It is
easy to see that $X_{lin}$ is itself $G$-linear if and only if all
its definably connected components are $G$-equivalent.

\noindent{\bf Notation} We denote by $X_{lin}^{max}$ the union of
all components of $X_{lin}$ of maximal dimension.
\\

\noindent{\bf Infinitesimals}

 It might be easier here to use the
language of infinitesimals: We move to an $|M|^+$-saturated
elementary extension $\CN$ of $\CM$. For $g\in G(\CM)$, we denote by
$\nu_g$ the intersection of all $\CM$-definable open neighborhoods
of $g$  (with respect to the group topology) in the structure $\CN$.
For $X\sub G$, we write $\nu_g(X)=\nu_g\cap X.$ Notice that for $X,Y\sub G$
 and $g,h\in G$, $\nu_g(X)=\nu_g( Y)$ iff $X=_gY$, and $\nu_g( X)-g=\nu_h( X)-h$
iff $X-g=_0 X-h$.

We will be using the following simple observation: If $X, Y$ are
$A$-definable sets, and $x$ is generic in $X$ and in $Y$ over $A$
then $\nu_x(X)\sub Y$.

\begin{lemma}\label{linear1}  $X$ is locally $G$-linear at $g\in X$ if and only if
$\nu_g(X)-g$ is a subgroup of $G$. In particular, $X$ is a $C^0$-manifold near $g$.
\end{lemma}
\begin{proof} : If $\nu_g(X)-g$ is a subgroup then for every $h\in
\nu_g(X)$, we have $h-g+\nu_g(X)=\nu_g(X)$. It easily follows that
$X$ is locally $G$-linear at $g$.

Conversely,  assume that $X$ is locally $G$-linear at $g$. We need
to prove: For all $h_1,h_2\in X$ sufficiently close to $g$ we have
$(h_1-g)-(h_2-g)=h'-g$, for some $h'\in X$, or equivalently,
$g+(h_1-h_2)\in X$.

We pick $h\in X$ close to $g$ and generic over $g$  such that
$X-g=_0 X-h$. We take $U\ni 0$, definable over
 parameters independent of $g,h$,   such that
 $(X-g)\cap U=(X-h)\cap U$. Because of the
genericity of $h$, there exists a neighborhood $V\ni h$, which we
may assume is contained in $h+U$, such that for all $h'\in X\cap V$,
$$(X-h)\cap U=(X-g)\cap U=(X-h')\cap U.$$ By adding $h$ to both sides
of the last equation, we obtain: For all $h'\in X\cap V$, $X\cap
(h+U)=(X-h'+h)\cap (h+U),$ hence (since $V\sub (h+U)$) also for all
$h'\in X\cap V$,
$$X\cap V=(X-h'+h)\cap V.$$

In particular, if $h'',h'\in X$ are sufficiently close to $h$ (so
that $h''-h'+h\in V$) then we have $(h''-h')+h$ in $X$. It follows,
as was pointed out above, that $\nu_h(X)-h$ is a subgroup of $G$.
Because $\nu_g(X)-g=\nu_h(X)-h$, it also follows that $\nu_g(X)-g$
is a subgroup of $G$.

Because the germ of $X$ at $g$ is $G$-equivalent to the germ of $X$
at some generic point near $g$, it follows that $X$ is a
$C^0$-manifold near $g$.
\end{proof}

\noindent{\bf Notation. } For a $G$-linear $X$, we denote by
$\nu(X)$ the group $\nu_x(X)-x$, for some (all) $x\in X$. Notice that
 $Y$ is a $G$-subset of $X$ iff
  $\nu_h(Y)-h$ is a subset of the group
 $\nu(X)$, for every $h\in Y$. Note also that two $G$-linear sets
  $X_1, X_2$ are $G$-equivalent if and only if
  $\nu(X_1)=\nu(X_2)$.

\begin{lemma}\label{sum1} Let $X,Y\sub G$ be definable sets of
dimension $d$. If $\dim(X+Y)=d$ then $X_{lin}$ is large in $X$,
$Y_{lin}$ large in $Y$ and $X_{lin}^{max}$, $Y_{lin}^{max}$ are
$G$-linear sets that are $G$-equivalent.
\end{lemma}
\begin{proof} We assume that $X$ and $Y$ are $\emptyset$-definable and write $Z=X+Y$.
Take $g$ generic in $X$ over $\emptyset$ and $h$ generic in $Y$
over $g$. Then $g+h$ is generic in $Z$ over each $g$ and $h$ and
hence for every two  $x,y\in \{g,h,g+h\}$ we have $\dim(x/y)=d$.

For every $\emptyset$-definable neighborhood $U$ of $g$, we have
$(U\cap X)+h\sub Z$. Moreover $g+h$ is generic in $(U\cap X)+h$
over $h$ and therefore $(U\cap X)+h$ contains $\nu_{g+h}(Z)$. It
follows that $\nu_g(X)+h\supseteq \nu_{g+h}(Z)$. By symmetry we
may conclude that $\nu_g(X)+h=\nu_{g+h}(Z)$, or equivalently,
$\nu_g(X)-g=\nu_{g+h}(Z)-(g+h)$.

In the very same way we see that $\nu_h(Y)-h=\nu_{g+h}(Z)-(g+h)$.
Hence, we showed that for {\em every} generic $g\in X$, $h\in Y$
which are independent from each other we have
$\nu_g(X)-g=\nu_h(Y)-h$.

Fixing a generic $h\in Y$, this implies that $\nu_{g'}(X)$ is
constant as $g'$ varies over all elements $g'$ which are generic
in $X$ over $h$. In particular, $X_{lin}$ contains all those
elements $g'$ and moreover all the components of $X_{lin}$ of
dimension $d$ are $G$-equivalent. The same argument shows that all
components of $Y_{lin}$ of dimension $d$ are $G$-equivalent to
each other and that $X_{lin}^{max}$ is $G$-equivalent to
$Y_{lin}^{max}$.
\end{proof}

\begin{cor}\label{sum2}
Let $X,Y,Z\sub G$ be definable sets of dimension $d$. Assume that
for every $g$ in a large set $X_0\sub X$, we have $\dim((g+Y)\cap
Z)=d$. Then $X_{lin}$ is large in $X$.
\end{cor}
\begin{proof} Without loss of generality, $X,Y,Z,X_0$ are
$\emptyset$-definable, and hence every generic element $g\in X$
belongs to $X_0$. For every such $g\in X_0$ there exists $c\in
(g+Y)\cap Z$ which is generic in $Z$ over $g$. It follows that
$h=c-g$ is in $Y$ and, as before, the dimension of any two of
$\{g,h,c\}$ is $2d$. We can now find open neighborhoods $V$ of $h$
and $W$ of $c$, defined over independent parameters, such that
$g+(V\cap Y)\sub (W\cap Z)$. Because of the genericity of $g$
there is a neighborhood $U$ of $g$ such that $(U\cap X)+(V\cap
Y)\sub (W\cap Z)$. In particular, $\dim((U\cap X)+(V\cap Y))=d$.
We can therefore apply Lemma \ref{sum1} to $U\cap X$ and $V\cap Y$
and conclude that $X$ is locally $G$-linear at $g$.\end{proof}

Note that we cannot conclude, under the assumption of the last
Lemma, that all components of $X_{lin}^{max}$ are $G$-equivalent.

\begin{definition}
If $U$ is an open symmetric ($U=-U$) neighborhood of $0$ and $Y\sub
G$ a definable set, we say that $h_1,h_2\in Y$ are {\em
$U$-connected in $Y$} if $h_2\in h_1+U$ and there exists a definable
path in $Y$ connecting $h_1$ and $h_2$, which is contained in
$h_1+U$.
\end{definition}

\begin{lemma}\label{linear2} Let $X$ be $G$-linear and $Y$ a $G$-subset of $X$.
 Take $g\in X$, and assume
that $U$ is an open symmetric neighbourhood of $0$ such that $(g+U)\cap X$ is
relatively closed in $g+U$. Then, for any $h_1,h_2\in Y$ that are
$U$-connected in $Y$, we have $g+(h_1-h_2)\in X$
\end{lemma}
\begin{proof} Because $h_1, h_2$ are $U$-connected in $Y$,
there exists a definable path $\gamma:[a,b]\to Y$ such that
$\gamma(a)=h_1, \gamma(b)=h_2$ and for every $t\in [a,b]$,
$\gamma(t)-h_1\in U$. Consider the set $T$ of all $t\in [a,b]$ such
that $g+(h_1-\gamma(t'))\in X$, for all $t'\leq t$.

 We claim that
$T$ is both open and closed (with respect to the order topology in $M$) in $[a,b]$. Indeed, because $(g+U)\cap
X$ is relatively closed in $g+U$, the set $T$ is closed in $[a,b]$.

To see that it is open, assume that $t_0\in T$. In particular,
$k=g+(h_1-\gamma(t_0))$ is in $X$. For  $t'\in[a,b]$ close to $t_0$
we have $(\gamma(t_0)-\gamma(t'))\in X-k$ (because
$\nu_{\gamma(t_0)}(Y)-\gamma(t_0)\subseteq \nu_k(X)-k=\nu(X)$ is a
subgroup of $G$), therefore $t'\in T$. It follows that $T$ is open and closed in $[a,b]$,
hence $T=[a,b]$ and therefore $g+(h_1-h_2)\in X$.\end{proof}

 The following technical lemma ensures that we can
extend every $G$-linear set and every $G$-subset of a $G$-linear set beyond
 its frontier (where the
frontier of $X$ is $Cl(X)\setminus X$).

\begin{lemma}\label{linear31}
 Assume that $X$ is $G$-linear and  $Y$ is  $G$-subset of $X$. Then
 $Cl(Y)$ is also a $G$-subset of $X$.
\end{lemma}
\begin{proof} Fix $h$ in $Cl(Y)\setminus Y$. We first prove that
$\nu_h(Y)-h\sub \nu(X)$. Namely we show: For all $h'\in Y$
sufficiently close to $h$ and every $g\in X$ we have $g+h'-h\in X$.

If the above fails then we have a curve $\gamma:[a,b)\to Y$, with
$\lim_{t\to b}\gamma(t)=h$ such that for all $t\in [a,b)$,
$g+\gamma(t)-h\notin X$. Fix $U$ an open symmetric  neighborhood of
$0$ such that $X\cap (g+U)$ is relatively closed in $g+U$. By
choosing $\gamma(a)$ sufficiently close to $h$, we may assume that
for every $t\in [a,b]$, we have $\gamma(t)\in \gamma(a)+U$ (so in
particular, $h\in \gamma(a)+U$). It follows that for every $t\in
(a,b)$, we have $\gamma(a)$ and $\gamma(t)$ are $U$-connected (as
witnessed by $\gamma$) and therefore, by Lemma \ref{linear2},
$g+\gamma(a)-\gamma(t)\in (g+U)\cap X$. Because $X\cap (g+U)$ is
closed in $g+U$, we may take $t$ to be $b$ and conclude that
$g+\gamma(a)-\gamma(b)=g+\gamma(a)-h\in X$, contradicting our
assumption.

We therefore showed that for all $h\in Cl(Y)\setminus Y$,
$\nu_h(Y)-h\sub \nu(X)$. By our assumption on $Y$ this is true for
every $h\in Cl(Y)$. Because $X$ is locally closed we may take the
closure on the left and conclude that $\nu_h(Cl(Y))-h\sub \nu(X)$
for every $h\in Cl(Y)$, hence $Cl(Y)$ is a $G$-subset of $X$.
\end{proof}

\begin{lemma}\label{linear32}  Let $X\sub G$ be a definable set. Assume that \\ (i) $X_{lin}$
is large in $X$ (i.e. $\dim(X\setminus X_{lin})<\dim(X)$) and \\
(ii) for every $h\in X$ and $g\in X_{lin}$, $X-g\sub_0 X-h.$

 Then
$X$ is $G$-linear.

\end{lemma}
\begin{proof}
 Assume that $\dim(X)=d$. First notice that by (ii), all the
components of $X_{lin}$ are $G$-equivalent to each other and
therefore $X_{lin}$ is itself a $G$-linear set of dimension $d$.

Because of (ii) the local dimension of $X$ at every point is $d$ and
therefore, by (i), $X_{lin}$ is dense in $X$. By Lemma
\ref{linear31}, $X$ is a $G$-subset of $X_{lin}$.

Given any $h\in X$, we have $$\nu_h(X)-h\sub \nu(X_{lin})\sub
\nu_h(X)-h.$$ (the left inclusion follows from the fact that $X$ is
 a $G$-subset of $X_{lin}$ while the right one is just assumption
 (ii)).

It follows that $X$ is locally $G$-linear at $h$ hence $X$ is
$G$-linear.
\end{proof}

Before the next lemma we make a small observation.
\begin{rk}\label{obs1} Let $Z,W\sub G$ be definable sets and let $f:Z\to
W$ be a definable continuous map (all are $0$-definable). Then for
every  $w\in f(Z)$ that is locally generic in $W$ and $z\in
f^{-1}(w)$, we have $\nu_w(W)=f(\nu_z(Z)).$ (By ``locally generic''
we mean that for some open $V\ni w$, we have
$dim(w/\emptyset)=\dim(V\cap W)$).
\end{rk}
\begin{proof}
Fix $V$ as above and let $U\subseteq G$ be a definable open
neighborhood of $z$ such that $w$ is still generic in $V\cap W$ over
the parameters (say $A$) defining $U$. Then $\dim(w/A)=\dim(w)=\dim
f(U\cap Z)$, $w\in f(U\cap Z)$ and therefore $f(U\cap Z)$ and $V\cap
W$ have the same germ at $w$. In particular, $\nu_w(W)=\nu_w(V\cap
W)\subseteq f(U\cap Z)$. Because this is true for every  $U$
neighborhood of $z$, it follows that $\nu_w(W)\subseteq
f(\nu_z(Z))$. The converse follows by continuity.
\end{proof}

\begin{lemma}\label{linear4}
Let  $X$ and $Y$ be two $G$-linear sets. Then:

\noindent (i) $X+Y$ is $G$-linear and we have
$$\nu(X+Y)=\nu(X)+\nu(Y).$$  In
particular, if $X$ and $Y$ are $G$-equivalent $G$-linear sets then
$X+Y$ is $G$-equivalent to them as well.

\noindent (ii) If $X$ and $Y$ are $G$-equivalent then $X\cup Y$ is a
$G$-linear set and $G$-equivalent to $X$ and $Y$.

\end{lemma}
\begin{proof}

(i) First, notice that, by continuity, for any $x\in X $, $y\in Y$,
and $z=x+y$ we have $\nu_x(X)+\nu_y(Y)\subseteq \nu_z(X+Y)$.  It
follows that for all $z\in X+Y$, we have $\nu(X)+\nu(Y) \subseteq
\nu_z(X+Y)-z$.

By  Remark \ref{obs1}, if $z=x+y$ is a generic element of $X+Y$ then
$\nu_z(X+Y)-z=\nu_x(X)+\nu_y(Y)-z=\nu(X)+\nu(Y)$.  Therefore, $X+Y$
is $G$-linear at every generic $z\in X+Y$ and we have
$\nu_z(X+Y)-z=\nu(X)+\nu(Y)$.

It follows that $(X+Y)_{lin}$ is large in $X+Y$ and the germs of $X+Y$ at all points in
$(X+Y)_{lin}$ are $G$-equivalent. Taken together with the above,
we see that $X+Y$ satisfies the assumptions of Lemma
\ref{linear32}(1), hence it  is $G$-linear and we have
$\nu(X+Y)=\nu(X)+\nu(Y)$.

If $X$ and $Y$ are $G$-equivalent then $\nu(X)=\nu(Y)$, hence
$\nu(X)+\nu(Y)=\nu(X)$, so $X+Y$ is $G$-equivalent to both $X$ and
$Y$.

(ii) It is easy to see that for all $z\in X\cup Y$, if $x\notin
Fr(X)\cup Fr(Y)$, then we have $(X\cup Y)-z=_0 X-z$ (if $z\in X$) or
$(X\cup Y)-z=_0 Y-z$ (if $z\in Y$). Because $X$ and $Y$ are
$G$-linear and $G$-equivalent, it follows that $(X\cup Y)_{lin}$ is
large in $X\cup Y$. Also, for every $x\in X$ we clearly have
$X-x\subseteq_0 (X\cup Y)-x$, and similarly for $y\in Y$. We
therefore can apply Lemma \ref{linear32}(1) again and conclude that
$X\cup Y$ is $G$-linear, and $G$-equivalent to $X$ and
$Y$.\end{proof}

\begin{lemma}\label{slinear3}
Let  $X$ be a $G$-linear set and $Y_1, Y_2$ two $G$-subsets of $X$.
Then

 \noindent (i) $Y_1+Y_2$ is a $G$-subset of $X$.

\noindent (ii) $Y_1\cup Y_2$ is a $G$-subset of $X$.

\end{lemma}
\begin{proof}

(i) Because $Y_1$ and $Y_2$ are $G$-subsets of $X$, for all $y_1\in
Y_1$, $y_2\in Y_2$, we have
$$(\nu_{y_1}(Y_1)+\nu_{y_2}(Y_2))-(y_1+y_2)=(\nu_{y_1}(Y_1)-y_1)+(\nu_{y_2}(Y_2)-y_2)
\subseteq \nu(X)+\nu(X)=\nu(X).$$

Again as before, if $z=y_1+y_2$ is a locally generic element of
$Y_1+Y_2$ then
$$\nu_z(Y_1+Y_2)-z=\nu_{y_1}(Y_1)+\nu_{y_2}(Y_2)-z\sub \nu(X)$$ and
hence $Y_1+Y_2$ is $G$-subset of $X$ at $z$. We thus have
$(Y_1+Y_2)_{slin}$ dense in $Y_1+Y_2$.
 By Lemma \ref{linear31}, $Y_1+Y_2$ is a $G$-subset
of $X$.

(ii) Here we only need to note that for $z\in Y_1\cap Y_2$, it is
not true in general that the germs of $Y_1$ and $Y_2$ at $z$
coincide. However, it is still true that $(Y_1\cup Y_2)-z\sub_0
X-g$, so we can proceed as before.
\end{proof}

We recall the following definition:
\begin{definition} Given a definable group $G$ in a  sufficiently saturated $\CM$,
a subgroup $\CH$ of $G$ is called {\em locally-definable}, if it
can be written as the directed union of definable sets
$\CH=\bigcup \{X_i:i\in I\}$, where $|I|< \kappa$.

We say that $\CH$ is {\em definably connected} if the $X_i$'s can
all be chosen to be definably connected.
\end{definition}
Such groups were sometimes called $\bigvee$-definable groups (see
\cite{00PPS2}) or Ind-definable groups (see \cite{07HPP}). The
dimension of a locally-definable group is taken to be $max\{\dim
X_i:i\in I\}$. Notice that if $\CH$ is definably connected then it
is actually definably path connected in the sense that any two
points can be connected by a definable path in $\CH$. The following
claim is easy to verify:

 \begin{claim} \label{Glin} If $\CH=\bigcup \{X_i:i\in I\}$ is a
locally-definable group and $g\in \CH$ then there exists $i\in I$
such that $g\in X_i$ and
$$\nu_g(\CH)=g+\nu_0(\CH)=\nu_g(X_i).$$
\end{claim}

\begin{lemma} Every locally-definable subgroup $\CH$ of $G$ can be
written as the directed union of $G$-linear sets.
\end{lemma}
\begin{proof} Without loss of generality all $X_i$'s in the definition of
$\CH$ have maximal dimension $d$. By Claim \ref{Glin}, for every
$g\in \CH$ there exists $i\in I$ such that
$\nu_g(\CH)=\nu_g(X_i)$. For such an $i$ we have:
$g\in(X_i)_{lin}$, the local dimension of $X_i$ at $g$ equals $d$
(hence, $g\in (X_i)_{lin}^{max}$) and all $(X_i)_{lin}^{max}$ are
$G$-linear and $G$-equivalent. If we  let $X_i'=(X_i)_{lin}^{max}$
then we have $\CH=\bigcup_{i\in I}X_i'$.

To see that $\{X_i':i\in I\}$ is a directed system of sets: Given
$i,j\in I$, we take $k\in I$ such that $X_i\cup X_j\sub X_k$ and
claim that $X_i'\cup X_j'\sub X_k'$. Indeed, if $g\in X_i'$ then
$\nu_g(X_i')=\nu_g(\CH)$ and hence $\nu_g(X_i)=\nu_g(X_k)$. In
particular, $g\in (X_k)_{lin}^{max}=X_k'$.\end{proof}

\begin{lemma}\label{linear5}
We assume \M\  is an $\omega$-saturated structure. Let $X$ be  a definable
$G$-linear set, $0\in X$. Then the following hold.

(1) The group $\lan X\ran $ generated
by $X$ is locally-definable and its germ at $0$ equals to
$\nu(X)$. In particular, $\dim(\lan X\ran)=\dim X$.

(2) If $Y$ a $G$-subset of $X$ containing $0$,
 then the group  generated by $Y$ is a locally-definable of dimension $\leq \dim X$,
 whose germ at $0$ is contained in $\nu(X)$.
\end{lemma}
\begin{proof}
(1) The group $\lan X\ran$ generated by $X$ is a countable
increasing union of the sets $X_0=X$, $X_1=X-X$, $X_2=(X-X)+(X-X),
\ldots$, so locally-definable. By the Lemma \ref{linear4}, each
$X_n$ is $G$-linear and $G$-equivalent to $X$ (so in particular, has
the same dimension as $X$). Given $g\in \lan X\ran $, there exists
$n\geq 0$ such that $g\in X_n$. Because of the $G$-linearity, for
every $k\geq n$, we have $X_n=_g X_k$. Because of saturation, there
exists a neighborhood $U$ of $g$ such that $U\cap \lan X\ran =U\cap
X_n$. It follows that the germ of $\lan X\ran$ at this point equals
to that of $X_n$ and in particular, $\lan X\ran_0=\nu(X)$.

(2) The group $\lan Y\ran $ generated by $Y$ is a countable
increasing union of the sets $Y_0=Y$, $Y_1=Y-Y$, $Y_2=(Y-Y)+(Y-Y),
\ldots$. Because $-Y$ is also a $G$-subset of $X$, we can apply
Lemma \ref{slinear3} and conclude that each $Y_n$ is a $G$-subset of
$X$ whose germ at $0$ is contained in $\nu(X)$. It follows that the
dimension of $\lan Y\ran$ is at most $\dim X$ and that the germ of
$\lan Y\ran$ at $0$ is contained in that of $\lan X\ran$.\end{proof}

We end this section with a small observation on locally-definable
subgroups:

\begin{lemma}\label{linear5.5} Let $\CH_1,\CH_2$ be locally-definable subgroups of $G$. Then $\CH_1+\CH_2$ is a locally-definable group whose germ at $0$ equals to the sum of the germs
of $\CH_1$ and $\CH_2$ at $0$.
\end{lemma}
\begin{proof} Let
$$\CH_1=\bigcup_{i\in I}X_i\,\, ;\,\, \CH_2=\bigcup_{j\in
J}Y_j.$$

As was pointed out above, we may assume that the $X_i$'s $Y_j$'s
are all $G$-linear. By Lemma \ref{linear4}(1), the sets $X_i+Y_j$
are all $G$-linear and $\nu(X_i+Y_j)=\nu(X_i)+\nu(Y_j)$. By Lemma
\ref{linear5}(1),
$$\nu_0(\CH_1+\CH_2)=\nu(X_i+Y_j)=\nu(X_i)+\nu(Y_j)=\nu_0(\CH_1)+\nu_0(\CH_2).$$
\end{proof}

\section{Definable sets containing abstract subgroups}
In this section \M\  will denote an $\omega$-saturated 
 o-minimal expansion of an ordered group.

\begin{thm}\label{abstgroups}
 Let $G$ be a definable abelian group, $\Gamma \subseteq  G$ an abstract
subgroup (i.e., $\Gamma$ is not necessarily definable). Let
$X\subseteq G$ be a definable set containing $\Gamma$ of minimal
dimension $d$. Then there exist a definable set $X'$ of dimension
$d$, a definably connected locally-definable subgroup $\CH$ of
dimension $d$, and $g_1,\ldots, g_k\in \Gamma$ such that
$\Gamma\subseteq X'\sub \bigcup_{i=1}^k \CH+g_i$.
\end{thm}
\begin{proof} Assume the nontrivial case $d>0$. 
 For every $g\in \Gamma$, the set $X\cap (g+X)$
contains $\Gamma$ and, hence by minimality, has dimension $d$. It
follows that the set $Y=\{g\in X: \dim(X\cap X+g)=d\}$ contains
$\Gamma$ and therefore, again by minimality, has dimension $d$.

We now apply Lemma \ref{sum2} to $Y,Y, X$ and $X$ (for $X_0,X,Y$ and
$Z$, respectively), and conclude that $Y_{lin}$ is large in $Y$.

Since $Y$ contains $\Gamma$ we may replace $X$ with $Y$ and assume
from now on that $X_{lin}$ is large in $X$ (however, it need not
be the case that $X$ or even $X_{lin}$ is $G$-linear). Moreover,
we pick $X\supseteq \Gamma$ such that $X_{lin}$ has the minimal
number of definably connected components of dimension $d$.  Note
that each component has infinitely many elements of $\Gamma$
(otherwise, we can replace it by finitely many points).

\noindent{\bf Claim 1} All the components  of $X_{lin}$ of
dimension $d$ are $G$-equivalent to each other.

\noindent{\bf Proof} Indeed, if $X_1, X_2$ are two such components
then for $g\in X_1\cap \Gamma$, the set $\{h\in X_2:g+h\in
X\}=X_2\cap X-g$ contains $X_2\cap \Gamma$ and therefore has
dimension $d$ (otherwise, we could replace $X_2$ by a definable
set of smaller dimension). For the same reason, the set of all
$g\in X_1$ such that $\dim(X_2\cap (X-g))=d$ has dimension $d$.
Just like in the proof of Corollary \ref{sum2}, we can apply
Lemma \ref{sum1} locally to $X_1$ and $X_2$, and conclude that for
some open $U$ and $V$ we have $U\cap X_1$ is $G$-equivalent to
$V\cap X_2$. Because $X_1$ and $X_2$ are $G$-linear it follows
that they are $G$-equivalent to each other. {\bf End of Claim 1}.
\\

We therefore showed that $X_{lin}^{max}$ is $G$-linear.
 It is left to handle $X^*=X\setminus X^{max}_{lin}$.
Fix one of the components $X_0$ of $X_{lin}^{max}$.

As before, for every $g\in \Gamma\cap X^*$, the set $\{h\in
X_0:g+h\in X\}$ contains $X_0\cap \Gamma$ and hence has dimension
$d$. Therefore, after possibly replacing $X^*$ by a smaller set we
may assume that for all $g\in X^*$ the set $(g+X_0)\cap X$ has
dimension $d$.

\noindent{\bf Claim 2} $X$ is a $G$-subset of $X_{lin}^{max}$.

\noindent {\bf Proof} By abuse of notation we let $\nu(X)$ be the
infinitesimal subgroup associated to the $G$-linear set
$X_{lin}^{max}$. By Lemma \ref{linear31}, it is enough to see that a
dense subset of $X$ is $G$-linear in $X_{lin}^{max}$. Namely, we
will show that for every locally generic $g\in X$, we have
$\nu_g(X)-g\sub \nu(X)$. It is clearly sufficient to consider $g\in
X^*$.

We fix an open set $U\sub G$ and $g\in X^*$ which is generic in
$U\cap X$. By our assumption on $X^*$, there exists $h$ generic in
$X_0$ over $g$ such that $g+h$ (generic) in $X$. As in the proof of
Lemma \ref{sum1}, it follows that $\nu_g(X)+h \sub \nu_{g+h}(X)$,
hence $$\nu_g(X)-g\sub \nu_{g+h}(X)-(g+h)=\nu(X).$$ (the right-most
equality follows from the fact that  $g+h$ is generic in $X$ and
hence belongs to $X_{lin}^{max}$). {\bf End of Claim 2.}

Let $X^{max}_{lin}= X_1\cup\cdots\cup X_r$ be the union of those
components of $X_{lin}$ of dimension $d$. By Lemma \ref{linear0}
each $X_i$ is $G$-linear.
 For each $X_i$, pick $g_i\in \Gamma\cap X_i$, and consider the
 set $X_i'=X_i-g_i$. By Claim 1, the union $X'=\bigcup_{i=1}^r X_i'$
is $G$-linear (and $\nu(X')=\nu(X_i)$ for every $i$). It also
contains $0$ and it is definably connected. Let $\CH'$ be the subgroup
of $G$ generated by $X'$. By Lemma \ref{linear5}(1), $\dim \CH'=\dim
X'=d$. We thus have
$$X^{max}_{lin}\sub \bigcup_{i=1}^r \CH'+g_i.$$

Let $X^1,\ldots, X^t$ be the definably connected components of
$X^*$. For every such  $X^j$, we may assume that $X^j\cap \Gamma\neq
\emptyset$ (for otherwise we may omit this component), take $g^j\in
\Gamma\cap X^j$, and let $X''=\bigcup_{i=1}^t X^j-g^j$. By Claim 2,
each $X^j$ (and therefore also $X^j-g^j$) is a $G$-subset of
$X^d_{lin}$. By Lemma \ref{slinear3} (2),  the set $X''$ is also a
$G$-subset of $X^{max}_{lin}$. It is also definably connected and
contains $0$.

 Let $\CH''$ be the subgroup of $G$ generated by $X''$. By Lemma
 \ref{linear5}(2), the germ of $\CH''$ at $0$ is contained in that
 of $X'$ and therefore of $\CH'$. We now let $\CH=\CH'+\CH''$. By
 Lemma \ref{linear5.5}, we have $\nu_0(\CH) =\nu_0(\CH'+\CH'')=\nu_0(\CH')$,
 so in particular the dimension of $\CH$ is $d$.

Putting the above facts together we obtain:
$$\Gamma\sub X\sub \bigcup_{i=1}^k\CH+g_i,$$ for some $g_1,\ldots,
g_k\in \Gamma$.

With that we end the proof of the theorem.\end{proof}

\begin{cor}\label{ctor}
Let $G$ be a definably connected abelian group. Let
$\Gamma$  be a divisible subgroup of the subgroup of torsion points $Tor(G)$. Let $X\subset G$ be a
definable set containing $\Gamma$. Then, there is a definably
connected locally-definable
 $\CH$ subgroup of $G$  with $dim\,\CH\leq dim\,X$ such
 that $\Gamma\subset \CH$.
\end{cor}
\begin{proof}
By  Theorem \ref{abstgroups}, taking $X$ of minimal dimension, there
is a definably connected locally-definable  group $\CH$ subgroup of
$G$, with $dim\,\CH=dim\,X$, and $g_1,\dots, g_k\in \Gamma$ such
that $\Gamma\subseteq \bigcup_{i=1}^k\CH+g_i$.
 Then $\Gamma\subset \CH$. Indeed,  let $g\in \Gamma$ and let
 $m= lcm(o(g_1),\dots,o(g_k))$, since $\Gamma$ is divisible
 there is $h\in\Gamma$ such that $g=mh$.  For such $h$  there is
 an $i$ such that $h= h'+g_i$, for some $h'\in \CH$. Then
 $g=mh=mh'\in\CH$.
\end{proof}

\section{The main result: commutative case}

In this section \M\ will be an  $\omega$-saturated  o-minimal
expansion of a real closed field and $G$ will denote  a definably compact definably
connected abelian group of dimension $n$. Such $G$ is divisible and hence the subgroup 
$Tor(G)$ and
 any $p$-Sylow of $G$ are
 also divisible  (the $p$-Sylow of $G$  is   $G_p=\bigcup_{n>0}G[p^n]$,
 where $G[m]=\{ g\in G\colon mg=0_G\}$).

By Theorem 1.1 in  \cite {04EO}  we have $G[m]\cong(\Z/m\Z)^n$ for any $m>0$, and $\pi_1(G)\cong \Z^n$, where  $\pi_1(G)$ is the {\em o}-minimal fundamental group of $G$.

Let $p$ be a prime number and let $x_1,\dots,x_n\in G[p]$  we say that $x_1,\dots,x_n$ are {\em $n$ independent $p$-torsion points} if they  are \gp-independent under the isomorphism $G[p]\cong\gp^n$.
 
\begin{thm}\label{large}
Let $X$ be a definable subset of $G$.
If $X$ contains a  $p$-Sylow of $G$, then $dim\,X=n$. \end{thm}

\begin{lem}\label{indep} Let $p$ be a prime number and let $x_1,\dots,x_n\in G$ be
$n$ independent $p$-torsion points. For each $i=1,\dots,n$ let $\tau_i$ be a path in
  $G$ from $0_G$ to $x_i$, and let $p\,\tau_i$ denote the loop at $0_G$, $t\mapsto p\,\tau_i(t)$.
   Then,  $[p\,\tau_1], \dots,[p\,\tau_n]\in\pi_1(G)$ are \Z-linearly independent.
\end{lem}
\begin{proof} Let $\varphi\colon \pi_1(G)\to G[p]$ the homomorphism $[\gamma]\mapsto
\tilde{\gamma}(1)$, where $\tilde{\gamma}$ is the unique path in $G$
starting at $0_G$ and such that $p\tilde{\gamma}=\gamma$ (see the proofs of propositions
 2.10 and 2.11 in \cite {04EO}). Suppose
$[p\tau_1], \dots,[p\tau_n]$ are
 \Z-linearly dependent and let
$m_1[p\tau_1]+ \cdots+ m_n[p\tau_n]=[k_{0_G}]$  with $(m_1,\dots,m_n)=1$, where $k_{0_G}$ is the constant loop at $0_G$. Applying
 $\varphi$
 to this equality we get $m_1x_1+\cdots+m_nx_n=0_G$ and hence $p|(m_1,\dots,m_n)$,
 a contradiction.
\end{proof}
\begin{lem}\label{summap}
Let $[\gamma_1],\dots,[\gamma_n]\in \pi_1(G)$ be \Z-linearly
independent. Then,
$$\left\{\sum_{i=1}^n \gamma_i(t_i)\colon t_i\in[0,1), 1\leq i\leq n\right\} =G.$$
\end{lem}
\begin{proof}
Consider the definably compact definably
 connected $n$-dimensional abelian group $\T=[0,1)^n$ and the definable
  map
$$ \begin{array}{ rccc}
      f\colon&\T&\longrightarrow&G   \\
      &(t_1,\dots,t_n)& \mapsto &  \sum_{i=1}^n \gamma_i(t_i).
\end{array} $$
It suffices to prove that $f$ is onto.  Since $f$ is continuous with respect
 the manifold topology of \T, we can see $f$ as a definable continuous map
  between definable manifolds.

\noindent{\bf Claim} It suffices to prove that $f$ induces an isomorphism
$f_*\colon H_1(\T;\Q)\to H_1(G;\Q)$.

\noindent{\bf Proof}  The map $f$ has a degree for any orientations of \T\
and $G$, since both \T\ and $G$ have dimension $n$ (see section 4 in \cite{04EO}).
 To prove that $f$ is onto it suffices to prove that one (equivalently each one) of
 these degrees is not zero. Let  $\zeta_G$ and $\zeta_{\T}$ be  the fundamental classes
  of some given orientations of $G$ and $\T$, i.e., $\zeta_G$ (respectively
   $\zeta_{\T}$) is a generator of $H_n(G)\cong \Z$ (resp. of  $H_n(\T) \cong \Z$).
   The degree of $f$ is defined by  $f_*(\zeta_{\T})=deg(f)\zeta_G$. If $\omega_G$
   and $\omega_{\T}$ are the corresponding cohomology classes by duality we have
   $f^*(\omega_G)= deg(f)\omega_{\T}$. Hence to prove that $deg(f)\not=0$ suffices
   to prove that $f$ induces an isomorphism of the \Q-vector spaces $f^*\colon H^n(G;\Q)\to H^n(\T;\Q)$.
   Now $f^*\colon H^*(G;\Q)\to H^*(\T;\Q)$ is also \Q-algebra homomorphism and by Theorem 1.1
    in \cite{04EO}, both $\Q$-algebras are   generated  by  elements of degree one.
     So  it suffices to prove that $f$ induces an isomorphism $f^*\colon H^1(G;\Q)\to H^1(\T;\Q)$.
     Applying duality again we have the required result. End of Claim.

Now, let $\delta_i\colon[0,1]\to \T\colon
t\mapsto\delta_i(t)=(0,\dots,\stackrel{i}{1},\dots,0)$. The map $f$ induces a map
 $f_*\colon \pi_1(\T)\to\pi_1(G)\colon [\delta_i]\mapsto  f_*([\delta_i])=[\gamma_i]$,
 for each $i=1,\dots,n$, which is one to one  and has finite cokernel  ($= \pi_1(G)/Im(f_*)$)
  because  we have $n$  \Z-linear independent  $[\gamma_i]$'s. Identifying the $\pi_1$'s
   with the $H_1$'s via the Hurewicz isomorphism, we have  the following exact sequence
$$ 0\to H_1(\T)\stackrel{f_*}{\to}H_1(G)\to H_1(G)/Im(f_*)\to 0.$$
tensoring with \Q\ we obtain that $f_*\colon H_1(\T;\Q)\to H_1(G;\Q)$ is an isomorphism.
\end{proof}
\begin{cor}\label{locdef}
Let $\CH$ be a definably connected locally-definable subgroup of
$G$. Suppose $\CH$ contains $G[p]$, for some prime $p$. Then
$\CH=G$.
\end{cor}
\begin{proof}Let $x_1,\dots,x_n$ be $n$ independent $p$-torsion points and
  for each $i=1,\dots,n$ let $\tau_i$ be a path in $\CH$ from $0_G$ to $x_i$.
   Let $\gamma_i$ denote the loop at $0_G$ defined by $t\mapsto p\tau_i(t)$,  since
   $\CH$ is a group $\gamma_i([0,1])\subset \CH$. By Lemma \ref{indep}
   $[\gamma_1], \dots,[\gamma_n]\in\pi_1(G)$ are \Z-linearly independent.

By Lemma \ref{summap}, $S=\left\{\sum_{i=1}^n \gamma_i(t_i)\colon
t_i\in[0,1),
 1\leq i\leq n\right\} =G$. Again since $\CH$ is a group $S\subset \CH$.
\end{proof}

\begin{proof}[Proof of Theorem \ref{large}]
Let $X\sub G$ be a definable set containing  a $p$-Sylow $G_p$ of $G$.
 Since $G_p$ is divisible, by Corollary
\ref{ctor}, there exists a definably connected locally-definable
group $\CH$ of dimension $\leq dim\,X$  containing $G_p$. Since
$\G_p$ contains $G[p]$, we can apply
 Corollary \ref{locdef} to get $\CH=G$ and hence $dim\,X=dim\,G$
\end{proof}

\begin{cor}\label{ncor} Let $X$ be a definable subset of $G$. If $X$ is large in $G$ then for any prime $p$ and any $l$ there is an $m>l$ and a $g\in X$ of order $p^m$.
 \end{cor}
 \begin{proof} Note that this corollary is equivalent to Theorem \ref{large}.  Indeed, for the nontrivial case suppose there is a prime $p$ and an $l$ such that $X\cap G_p\subset G[p^l]$. So the set $X'=X\setminus G[p^l]$ is still large in $G$ and $X'\cap G_p=\emptyset$. But then $G\setminus X'$ contains $G_p$ and $\dim(G\setminus X')<\dim G$, a contradiction with the theorem.
 \end{proof}

\section{The main result: general  case}
We work in an $\omega$-saturated structure  $\CM$ which is an o-minimal expansion of a real closed field $\CM$.
 We will use multiplicative notation for groups, and denote
  by $p_k$ the map $x\mapsto x^k$. A definably connected group 
   which is either abelian or definably compact is divisible (see \cite{06B}
    or \cite{05E1}), hence  the map $p_k$ is onto. Recall that a {\em generic subset}
     $Y$ of a definable abelian group $G$ is a definable subset of $G$ such that
      finitely many translates of $Y$ cover $G$.
\begin{lem}\label{covT}
Let $T$ be  a definable  abelian group. Suppose that for every $Y\sub T$ generic,
 $Y\cap Tor(T)\not=\emptyset$. Then, for each  $Y\sub T$ generic there is a $k\in\Na$
  such that $p_k(Y)=T$.
\end{lem}
\begin{proof}
First note that for each  $Y\sub T$ generic, there are $g_1,\dots,g_l\in Tor(T)$
 such that $T=\bigcup_{i=1}^ng_iY$.
Indeed,  if  not,  by compactness, there is $h\not\in Tor(T)Y$,
i.e., $hY^{-1}\cap Tor(T)=\emptyset$,
 but $hY^{-1}$ is also generic, a contradiction.
 Now taking $k=lcm(o(g_1),\dots,o(g_l))$ we have $p_k(Y)=p_k(T)=T$ as required.
\end{proof}

We recall that  a maximal definably connected abelian subgroup $T$
of a definably connected  definably compact group $G$ is called a
{\em maximal torus of $G$} and that $G=\bigcup_{g\in G}T^g$ (see
Theorem 6.12 in \cite{06B} or \cite{05E1}).

\begin{lem}\label{covG}
Let $G$ be  a definably compact definably connected  group. Let $T$
be a maximal torus of $G$ and $X\sub G$ is definable. Assume the
following:

(i) For every $g\in G$ and for every   $Y\sub T^g$ generic, $Y\cap
Tor(T^g)\not=\emptyset$, and

(ii) For every $g\in G$, $X\cap T^g$ is generic in $T^g$.

 Then there is a
$k\in\Na$ such that $p_k(X)=G$.
\end{lem}
\begin{proof} Given $g\in G$, by (ii), we can apply
 Lemma \ref{covT}, for $Y=X\cap T^g$, to get
 a $k(g)\in \Na$ such that the map $y\mapsto y^{k(g)}$ sends $X\cap T^g$ onto $T^g$
(note that by divisibility of $T^g$  any multiple of $k(g)$ has the same property).
 Then,  by compactness, there is $k\in\Na$ such that for every $g\in G$,
the map $y\mapsto y^k$ sends $X\cap T^g$ onto $T^g$.
 To finish the proof  suffices to make use of  the equality $G=\bigcup_{g\in G}T^g$.
\end{proof}

For the next lemma we will use the notions of {\em compact domination} and
{\em very good reduction} introduced in  \cite{07HPP}. Recall that  a definably
simple group has very good reduction and that a definably compact   group which has
very good reduction has also compact domination (see Theorem 10.7 in \cite{07HPP}).

\begin{lem}\label{vgr}
Let $G$ be  a definably compact definably connected  group. Let $X\sub G$ be a definable set such that $Tor(G)\sub X$. If $G$ has very
 good reduction then there
  is $k\in\Na$ such that $p_k(X)=G$.

  Moreover, if $N$ is a finite (central) normal subgroup of $G$ and 
  $G/N$ has very good reduction  then there is $k\in\Na$ such that
  $p_k(X)=G$.
\end{lem}
\begin{proof}  The result can probably be read off
standard Lie theory. However, since we could not find a reference we
give a complete proof.

Without lost of generality we may assume that $G=G(M)$ is defined
over the reals. By our assumptions on $G$, $G(\R)$ is a compact
connected Lie group. Let
 $T_0$ be a standard maximal torus of $G(\R)$, hence $T_0$ is also a Lie group and hence definable
  over the reals. Let $T=T_0(\CM)$. So  $T$ is still a maximal torus of $G$ and it has very
  good reduction and hence it  is  compactly dominated. The same happens for every conjugate
   $T^g$ of $T$, for $g\in G$. Since all $T^g$'s are also abelian, we can apply
    Proposition 10.6 in  \cite{07HPP}, to get that  for every $g\in G$, every generic subset
     of $T^g$ has a torsion point, so condition (i) of Lemma \ref{covG} is satisfied.

      Because $Tor(G)\sub X$,  the set  $T^g\setminus  X$ has no
      torsion and hence, by (i), it is non-generic. It follows that $T^g\cap X$
      is generic and therefore (ii) holds.   Now
       we can apply Lemma \ref{covG} to get the required result.

       For the moreover clause, let $\pi:G\to G/N$ be the projection
       map. Because $N$ is finite, all torsion elements of $G/N$ are
       in $\pi(X)$, hence by what we have just proved, there is $k$
       such that $p_k(\pi(X))=G/N$. If we take $k'=k\cdot|N|$ then 
       $p_{k'}(X)=G$ (because  $N$ is central and $G$ divisible).
\end{proof}

Notice, in the setting of the above lemma, that if $p_k(X)=G$ then
in particular, $\dim X=\dim G$ (since the image of $X$ under a
definable map cannot increase).

 We can now
prove the main result:
\begin{thm}\label{general} Let $G$ be a definably compact group
and assume that $X\sub G$ is a definable set containing all torsion
points of $G$. Then $\dim X=\dim G$.
\end{thm}
\begin{proof} By 4.1 in \cite{00PPS1} and 5.4 in \cite{00PS},
after modding out $G_1=G/Z(G)$ by its finite center, the group we
obtain is a direct product of simple groups. In particular, it has
very good reduction (see the proof of 5.1 in \cite{02PPS}) and hence
the ``moreover'' clause of Lemma \ref{vgr} holds for $G/Z(G)$. We
therefore have:

(i) For every definable $X_1\sub Z(G)$, if $Tor(Z(G))\sub X_1$ then
$\dim X_1=\dim Z(G)$ (by Theorem \ref{large} this is true for
subsets of $Z(G)^0$, but then it clearly follows for subsets of
$Z(G)$ as well).

(ii) For every definable $X_1\sub G$, if $Tor(G/Z(G))\sub X_1/Z(G)$
then $\dim(X_1/Z(G))=\dim(G/Z(G)).$ (by Lemma \ref{vgr} and our
previous observation).

We now proceed as follows, with $X\sub G$ a definable set containing
$Tor(G)$.  Given $g\in Tor(G)$ and $h\in Tor(Z(G))$, we have $gh\in
Tor(G)\sub X$. Hence, for every such $g$, we have $$X_g=\{h\in
Z(G):gh\in X\}\supseteq Tor(Z(G)),$$ so, by (i), $\dim X_g=\dim
Z(G)$. Since $gX_g= (gZ(G))\cap X$, we have, for every $g\in
Tor(G)$,
$$\dim gZ(G)\cap X=\dim Z(G).$$

 Let $$X_1=\{g\in G:\dim (gZ(G)\cap X)=\dim Z(G)\}.$$
It follows from the above that $Tor(G)\sub X_1$, and therefore
$Tor(G/Z(G))\sub X_1/Z(G)$. (Indeed, if $hZ(G)\in Tor(G/Z(G))$ then for
some $n\in \Na$, $h^n\in Z(G)$, and therefore for some $k\geq n$,
$h^k\in Z(G)^0$. Because $Z(G)^0$ is divisible there exists $h_1\in
Z(G)$, with $h_1^k=h^k$. Because $h_1$ is central,
$(hh_1^{-1})^k=e$, hence $hh_1^{-1}\in X_1$ and therefore $hZ(G)\in
X_1/Z(G)$). By (ii), $\dim (X_1/Z(G))=\dim (G/Z(G)).$

To finish the proof, consider the set $Y=(X_1Z(G))\cap X$. The
cosets of $Z(G)$ partition $Y$ into equivalence classes, each of
dimension
 $\dim Z(G)$ (by definition of $X_1$).  Since every $Z(G)$-coset of
 an element in $X_1$ intersects $X$
 nontrivially,
  we have $$\dim(Y/Z(G))=\dim(X_1/Z(G))=\dim(G/Z(G)).$$
Summarizing, we have
$$\dim Y\geq \dim Z(G)+\dim(G/Z(G))=\dim G,$$ and hence (since
$Y\sub X$) $\dim X=\dim G$.\end{proof}

\end{document}